\documentclass{article}
\usepackage[utf8]{inputenc}
\usepackage{ifluatex}
\usepackage{pdfpages}
\usepackage[T1]{fontenc}
\usepackage{lmodern}
\usepackage{etoolbox}
\usepackage{amsmath,amssymb,amsthm} 
\usepackage{cmap}
\usepackage{mathrsfs}
\usepackage{verbatim}
\usepackage{lastpage}   
\usepackage{indentfirst}  
\usepackage{color}      
\usepackage{graphicx}   
\usepackage{microtype}
\usepackage{filecontents}
\usepackage{enumerate}
\usepackage{color}
\usepackage{verbatim}

\newcommand{\R}{\mathbb{R}}

\newtheorem{theorem}{Theorem}[section] 
\newtheorem{lemma}[theorem]{Lemma}     
\newtheorem{corollary}[theorem]{Corollary}

\newtheorem{remark}{Remark}

\newtheorem{example}{Example}

\title{Complete stationary spacelike surfaces in an $n$-dimensional Generalized Robertson-Walker spacetime}
\author{Danilo Ferreira, Eraldo A. Lima Jr. and Alfonso Romero}
\date{27 August 2020}

\begin{document}
	\maketitle
	
	\begin{abstract}
		Several uniqueness results for non-compact complete stationary spacelike surfaces in an $n(\geq 3)$-dimensional Generalized Robertson Walker spacetime are obtained. In order to do that, we assume a natural inequality involving the Gauss curvature of the surface, the restrictions of the warping function and the sectional curvature of the fiber to the surface. This inequality gives the parabolicity of the surface. Using this property, a distinguished non-negative superharmonic function on the surface is shown to be constant, which implies that the stationary spacelike surface must be totally geodesic. 
		Moreover, non-trivial examples of stationary spacelike surfaces in the four dimensional Lorentz-Minkowski spacetime are exposed to show that each of our assumptions is needed.\\
	\end{abstract}

	\noindent Keywords: Stationary surfaces, Parabolic Riemannian surfaces, Generalized Robertson-Walker spacetimes.
	\footnote{53C43 (primary), 53A10, 53C17, 53Z05  (secondary)}

	\section{Introduction} 
	\label{intro}
	Spacelike surfaces with zero mean curvature in an $n(\geq 3)$-dimensional spacetime are surfaces whose induced metric is Riemannian  and they are locally critical points of the area functional. In the  $3$-dimensional Lorentzian setting, the expression maximal is more utilized than stationary because of the fact that such a surface is a local maximum of the area functional in relevant spacetimes (see \cite{MT:80} for instance). Agreeing with \cite{AB:98}, for ambient spacetimes of dimension greater than three, the term``stationary'' is more suitable for a spacelike surface with zero mean curvature. Stationary spacelike surfaces in $4$-dimensional spacetimes are a relevant role in mathematical Relativity. Indeed, a stationary spacelike surface may be seem as a limit of trapped surfaces which are inside the horizon of events around a singularity. Trapped surfaces are usually considered to be compact. We are interested here in parabolic stationary surfaces, i.e., stationary spacelike surfaces $\Sigma^2$ such that  the only nonnegative superharmonic functions on $\Sigma^2$ are the constants. Obviously, this property is satisfied in the compact case. The study of stationary spacelike surfaces has always been subject for many researchers in the past decades. For instance, in \cite{AA:09} Al\'ias and Albujer proved that in a Lorentzian product $-\mathbb{R}\times M^2\equiv (\mathbb{R}\times M^2,-dt^2+g_M)$ with Gauss curvature of $M^2$ satisfying $K_M\geq 0$, any complete maximal surface is totally geodesic, moreover if $M^2$ is not flat then these surfaces are just the slices $\{t_0\}\times M^2$. Further in \cite{AA:11} they developed their results for surfaces with non-empty boundary. We also observe that in \cite{A:08} Albujer constructed maximal surfaces in $-\mathbb{R}\times \mathbb{H}^2$, where $\mathbb{H}^2$ is the hyperbolic plane of constant Gauss curvature $-1$, in order to justify these curvature restriction on the Riemannian surface $M^2$. Moreover, new uniqueness properties of complete maximal surfaces in Lorentzian product spacetimes $-I\times M^2$, where $K_M\geq -\kappa$ for some constant $\kappa >0$ were obtained by the second and the third authors in \cite{LR:16}, by means of an extension of a well-known result by Nishikawa in \cite{Nis}. In \cite{RR:10} the third author and Rubio proved several uniqueness results for complete maximal surfaces in $-\mathbb{R}\times_f\mathbb{R}^2\equiv (\mathbb{R}\times \mathbb{R}^2,-dt^2+f(t)^2g_0)$ which is a $3$-dimensional Robertson-Walker spacetime with fiber the Euclidean plane $(\mathbb{R}^2,g_0)$ and warping function $f:\mathbb{R} \to \mathbb{R}^+$ extending the uniqueness results of Latorre and the third author in \cite{LR:01}. The results in \cite{AA:09}
	were extended to $3$-dimensional GRW spacetimes by Caballero, the third author and R.M. Rubio in \cite{CRR}. 
	
	It was proved by Al\'ias, Estudillo and the third author in \cite{AER:96} that the only compact stationary spacelike surfaces $\Sigma^2$ in a $3$-dimensional GRW spacetime $-I\times_{f}M$ with Gauss curvature satisfying 
	\begin{equation}\label{eq I}
		K_{\Sigma}\geq \frac{f'(\tau)^{2}+K_{M}}{f(\tau)^{2}},
	\end{equation}
	where $\tau$ is the restriction of the projection $\pi_{I}$ to $\Sigma^2$ and $K_{M}$ stands for the sectional curvature of $M$ restricted to $\Sigma^2$, are the totally geodesic ones. This result was obtained using a universal integral inequality whose equality characterizes the totally case, \cite[Theorem 9]{AER:96}. As the Example \ref{Ex1} below shows, the assumption \eqref{eq I} does not imply the same conclusion if the hypothesis compact is changed to complete, in general. Thus, the following question arises in a natural way,
	\begin{quote} {\it
			Under what assumptions a (non-compact) complete stationary spacelike surface in a GRW spacetime with Gauss curvature satisfying \eqref{eq I} must be totally geodesic? }
	\end{quote} 
	\vspace{2mm}   
	In this paper we will give several answers to this question, focusing our attention in natural assumptions on the  Gauss curvature, that lead to the parabollicity of the spacelike surface, and a remarkable smooth function naturally defined on the surface, that satisfies a certain partial differential equation from the fact that the mean curvature vector field vanishes everywhere.

	\section{Preliminaries}
	Let $(M^{n-1},g)$ be an $(n-1)$-dimensional Riemannian manifold and  let $I\subset \R$ be an open interval of the real line $\R$. An $n$-dimensional Generalized Robertson-Walker ($GRW$) spacetime  $\overline{M}^{n}:=-I\times_{f} M^{n-1}$is the warped product with base $(I,-dt^2)$, fiber $(M^{n-1},g)$  and warping function $f:I\rightarrow \R^{+}$. Thus, it is a time orientable Lorentzian manifold with the metric
	\begin{equation}\label{metric}
		\langle\ ,\ \rangle = -\pi_I^*dt^{2}+ f^{2}(\pi_I)\pi_M^*g,
	\end{equation}
	where, as usual, $\pi_I$, $\pi_M$ denote the corresponding projections onto $I$, $M^{n-1}$ respectively. Thus, $\overline{M}^{n}$ becomes a spacetime when it is endowed of the time orientation defined by the timelike vector field $\partial / \partial t$ \cite{ARS:95}. In the case $f=1$ the GRW spacetime is called static and it is denoted by $-I\times M^{n-1}$, in other words, a static GRW spacetime is the Lorentzian product of a negative definite open interval and a Riemannian manifold.
	
	Let $x:\Sigma^2\rightarrow \overline{M}^{n}$ be a spacelike surface\footnote{Any spacelike surface will be assumed to be connected.}, i.e, $x$ is an immersion and the induced metric $x^{*}\langle\ ,\ \rangle$ on $\Sigma^2$ is Riemannian. The timelike vector field $T := f(\pi_I)\partial / \partial t\in \mathfrak{X}(\overline{M}^{n}),$ decomposes along $x$ as follows
	\begin{equation}\label{T}
		T = T^{\top}+T^{N}
	\end{equation}
	where, at any point of $\Sigma^2$, $T^{\top}$ is the  tangent component  and $T^{N}$ is the normal component of $T$. 
	The vector field $T$ satisfies $\overline{\nabla}_{X}T = f'(\pi_I)X$, for all $X\in \mathfrak{X}(\overline{M}^n)$, where $\overline{\nabla}$ is the Levi-Civita connection of the Lorentzian metric \eqref{metric}, i.e, it is conformal with $\mathcal{L}_{T}\langle\ ,\ \rangle = 2f'(\pi_I)\langle\ ,\ \rangle$ and closed. This property on $\overline{M}^{n}$ is  translated, via the Gauss and Weingarten formulas (see \cite[Chapter 4]{O'N:83} for instance), to $\Sigma^2$ as it follows
	\begin{equation}\label{conn}
		\nabla_{X}T^{\top} = A_{T^{N}}X+f'(\tau)X,\ \ \nabla^{\perp}_{X}T^{N} = -\sigma(T^{\top},X)
	\end{equation}
	where $X\in \mathfrak{X}(\Sigma)$, $\nabla$ is the Levi-Civita connection of the induced metric, which is denoted by the same symbol as in \eqref{metric}, $A_{T^N}$ is the Weingarten endomorphism corresponding to $T^{N}\in \mathfrak{X}^{\perp}(\Sigma),$ $\tau:=\pi_I \circ x$, $\nabla^{\perp}$ is the normal connection and $\sigma$ is the second fundamental form.
	
	Let us consider the smooth function $u:= -\langle T^{N},T^{N}\rangle = f(\tau)^{2}+|T^{\top}|^{2}\geq f(\tau)^{2}>0$ on $\Sigma^2.$ From the first equation in \eqref{conn} we get the following expression for the gradient of $u$, $\nabla u = 2A_{T^{N}}T^{T}$. Therefore 
	\begin{equation}\label{grad u}
		|\nabla u|^{2} = 4\langle A^{2}_{T^{N}}T^{\top},T^{\top}\rangle.
	\end{equation}
	In the case the mean curvature vector field of $\Sigma^2$ vanishes identically, we call $\Sigma^2$ a stationary spacelike surface. Under previous assumption, previous formula is written as
	\begin{equation}\label{gradu2}
		|\nabla u|^{2}= 2 \text{tr}(A^{2}_{T^{N}})|T^{\top}|^{2} ,
	\end{equation}
	because in this case we have $A^{2}_{T^{N}}-\frac{1}{2}\text{tr}(A^{2}_{T^{N}})I = 0$, where $I$ is  the identity transformation.
	
	As shown in \cite{AER:96}, from the Gauss and Codazzi equations, when $\textbf{H} = 0$, formulas \eqref{conn} and the expression of the curvature 
	tensor $\overline{R}$ of $\overline{M}^{n}$  in terms  of the warping function and the curvature of $M$ \cite[Prop. 7.42]{O'N:83}, we  get for the Laplacian of $u$
	
	\begin{equation}\label{Delta u}
		\Delta u  = 2\left(K_{\Sigma} - \frac{f''(\tau)}{f(\tau)}\right)|T^{\top}|^{2}+ 2\text{tr}(A^{2}_{T^{N}}).
	\end{equation}
	A direct computation from \eqref{gradu2} and \eqref{Delta u} gives
	
	\begin{lemma}\label{Deltalog}
		Let $x:\Sigma^2\rightarrow \overline{M}^{n}$ be a stationary spacelike surface. For the function $u=-\langle T^{N},T^{N}\rangle>0$ on $\Sigma^2$ we have 
		\begin{equation}
			\Delta \log u \!=\! 2\left(K_{\Sigma}-\frac{f''(\tau)}{f(\tau)}\right)-\frac{2f(\tau)^{2}}{u^2}\left\lbrace \left(K_{\Sigma}-\frac{f''(\tau)}{f(\tau)}\right)u-\text{tr}(A^{2}_{T^{N}})\right\rbrace .
		\end{equation}
		
	\end{lemma}
	In the order to go further, around any $p\in \Sigma^2$ consider a local orthonormal normal frame $\{\xi_{1},...,\xi_{n-2}\}$ where $\xi_{n-2}$ is, at any point, collinear to $T^{N}$. Thus, we have $\langle\xi_{i},\xi_{j}\rangle = \delta_{ij},\ \ 1\leq i,j\leq n-3,\ \langle\xi_{n-2},\xi_{n-2}\rangle = -1$. Under the assumption that  $\Sigma^2$ is stationary, the Gauss equation becomes
	\begin{equation}\label{Gauss eq}
		2K_{\Sigma} =2\overline{K} - \sum_{i=1}^{n-3}\text{tr}(A^{2}_{\xi_{i}})+\frac{1}{u}\text{tr}(A^{2}_{T^{N}}) 
	\end{equation}
	where $\overline{K}$ is, at any point $q\in \Sigma^2$, the sectional curvature in $\overline{M}^{n}$ of the spacelike tangent plane $\text{d}x_{q}(T_{q}\Sigma)\subset T_{x(q)}\overline{M}^{n}.$ Next, $\overline{K}$ may be expressed in terms of the warping function and the sectional curvature $K_{M}$ of the fiber as it follows \cite[ Lemma 2]{AER:96}.
	
	\begin{lemma}\label{lem K}
		Let $\Sigma^2$ be a spacelike surface in $\overline{M}^{n}$. Then the sectional curvature in $\overline{M}^{n}$ of the spacelike plane tangent to $\Sigma^2$ is given by \[\overline{K} = \frac{f''(\tau)}{f(\tau)}+\frac{f'(\tau)^{2}-f''(\tau)f(\tau)}{f(\tau)^4}u+\frac{u}{f(\tau)^4}K _ {M}.\]
	\end{lemma}
	Now using the previous result, the Gauss equation \eqref{Gauss eq} can  be rewritten to get
	\begin{align*}
		K_{\Sigma}  =\frac{f''(\tau)}{f(\tau)}+\frac{f'(\tau)-f''(\tau)f(\tau)}{f(\tau)^{4}}u+\frac{u}{f(\tau)^{4}}K_{M}-\frac{1}{2}\sum_{i=1}^{n-3}\text{tr}(A^{2}_{\xi_i})+\frac{1}{2u}\text{tr}(A^{2}_{T^{N}}).
	\end{align*}
	Taking into account that the gradient on $\Sigma^2$  of $\tau$ satisfies $\nabla \tau =-\partial_{t}^{\top}$, we obtain $|\nabla \tau|^{2} = (u-f(\tau)^2)/f(\tau)^{2}$. Therefore,
	\begin{align}\label{KS}
		K_{\Sigma} = &\frac{f'(\tau)^{2}+K_{M}}{f(\tau)^{2}}+\left\lbrace\frac{K_{M}}{f(\tau)^{2}}-(\log f)''(\tau)\right\rbrace|\nabla \tau|^{2}+\frac{1}{2u}\text{tr}(A^2_{\xi_i})\nonumber\\
		&-\frac{1}{2}\sum_{i=1}^{n-3}\text{tr}(A^{2}_{\xi_i})-\frac{f''(\tau)}{f(\tau)}.
	\end{align}
	
	Now we can rewrite the conclusion of Lemma \ref{Deltalog} with the help of the formula \eqref{KS} as it follows.
	\begin{lemma}\label{Deltalog2}
		Let $x:\Sigma^2\rightarrow \overline{M}^{n}$  be a stationary spacelike surface. For the function $u = -\langle T^{N},T^{N}\rangle>0$ on $\Sigma^2$ we have
		\begin{equation}\label{Deltalog eq}
			\Delta \log u = 2\left(K_{\Sigma}-\frac{f'(\tau)^{2}+K_{M}}{f(\tau)^{2}}\right)+\frac{f(\tau)^{2}}{u^2}\left\lbrace \text{tr}(A^{2}_{T^{N}})+u\sum_{i=1}^{n-3}\text{tr}(A^{2}_{\xi_{i}})\right\rbrace.
		\end{equation}
	\end{lemma}
	Note that in the case $u$ is constant, formula \eqref{Deltalog eq} gives 
	\begin{equation}\label{eq K}
		K_{\Sigma}-\frac{f'(\tau)^2+K_{M}}{f(\tau)^{2}}\leq 0
	\end{equation}
	everywhere on $\Sigma^2$, and the equality holds in \eqref{eq K} if, and only if, $A_{T^{N}} = A_{\xi_{i}} = 0,\ i = 1,2,...,n-3$, i.e., if, and only if, $\Sigma^2$ is totally geodesic in $\overline{M}^{n}$.
	
	\section{Curvature and Parabolicity}
	A (non-compact) $n(\geq 2)-$dimensional Riemannian manifold $(S,g)$ is said to be parabolic if the only nonnegative superharmonic functions on $S$ are the constants.
	When $n=2$, this notion is very close to the classical parabolicity for Riemann surfaces. Moreover, it is strongly related to the behavior of the Gauss curvature $K$ of the Riemannian surface $(S,g)$. Thus, a classical result by Ahlfors and Blanc-Fiala-Hubber \cite{K:87} asserts that if $K\geq 0$, then a complete Riemannian surface must be parabolic.
	
	If $B_r$ and $B_R$, $0<r<R$, denote geodesic balls centered at the same point of a Riemannian manifold $(S,g)$ with $\dim S\geq 2,$ the quantity \[\frac{1}{\mu_{r, R}}:= \int_{A_{r,R}}|\nabla\omega_{r,R}|^{2} dS\]
	is called the capacity of the annulus $A_{r,R}:= B_{R}\setminus \overline{B}_r,$ where $\omega_{r,R}$ is the harmonic measure of $\partial B_{R}$ with respect of the problem
	\begin{equation}
		\Delta\omega = 0\ \ \text{in}\ \ A_{r,R},\ \omega  =0\ \ \text{on}\ \ \partial B_r\ \ \text{and}\ \ \omega = 1\ \ \ \text{on}\ \  \partial B_{R}.
	\end{equation}
	It is well known that a complete Riemannian manifold is parabolic if, and only if, for any fixed arbitrary point of $S$ we have
	\[\lim_{R\rightarrow \infty}\frac{1}{\mu_{r,R}} = 0.\]
	independent of $r$.
	
	To end this section, we recall the following technical result \cite[Lemma 2]{RR:10}.
	
	\begin{lemma}\label{Lem 2.2}
		Let $S$ be an $n(\geq 2)$-dimensional Riemannian manifold and let $v\in\mathcal{C}^{2}(S)$ that satisfies $v\Delta v \geq 0$. Let $B_R$ be a geodesic ball of radius $R$ in $S$. For any $r$ such that $0 < r < R$ we have
		\[\int_{B_r}|\nabla v|^{2}dV\leq \frac{4\sup_{B_{R}}v^{2}}{\mu_{r,R}}, \]
		where $B_r$ denotes the geodesic ball of radius $r$ centered at $p\in M$ and $\dfrac{1}{\mu_{r,R}}$ is the \textit{capacity} of annulus $B_{R}\setminus \overline{B}_{r}$.
	\end{lemma}
	
	\section{Main results}
	We will obtain below several uniqueness theorems using the previous results on parabolicity.
	The first result establishes sufficient conditions for a complete stationary surface to be totally geodesic.
	
	\begin{theorem}\label{teo 3.2}
		Let $x:\Sigma^2\to \overline{M}^{n}$ be a complete stationary spacelike surface in a GRW spacetime $\overline{M}^{n}=-I\times_fM^{n-1}$. If     
		\vspace{1mm}
		\begin{equation}\label{Kgeqzero}
			K_\Sigma\geq \frac{f'(\tau)^{2}+K _ {M}}{f(\tau)^{2}}\geq 0 
		\end{equation}
		and the function $u = - \langle T^{N},T^{N}\rangle$ on $\Sigma^2$ satisfies   
		\begin{equation}\label{subquadratic}  
			u \leq Cf(\tau)^2+C,  
		\end{equation} 
		\vspace{1mm}
		for some positive constant $C$, then $\Sigma^2$ must be totally geodesic.
	\end{theorem} 
	
	\begin{proof}
		Recall first of all the well-known formula $\Delta \phi(h) = \phi'(h)\Delta h+\phi''(h)|\nabla h|^2$, that holds true for  any smooth function $h$ on $\Sigma^2$ and $\phi$ a real function two times differentiable. If we put $h=u$, $\phi(u)=\log u$ and $\Delta \log u=\psi$, we can write  
		\begin{equation}\label{new}  
			\Delta u=\psi u+\frac{|\nabla u|^2}{u}.
		\end{equation}    
		In order to obtain the conclusion let us define the auxiliary function   
		
		\[
		h(t)=\frac{1}{(1+t)^\alpha},
		\]
		$t>0$, for each positive constant $\alpha$. We have    
		
		\begin{align*}
			h'(t)&=\frac{-\alpha}{(1+t)^{\alpha+1}},\\[1mm]
			h''(t)&=\frac{\alpha(\alpha+1)}{(1+t)^{\alpha+2}},
		\end{align*}
		and
		\[
		\frac{h''(t)}{h'(t)} =-\frac{\alpha+1}{1+t}.
		\]
		Consider $h(u)$ and note that $h(u)< 1$, i.e., $h(u)$ is bounded from above. Using \eqref{new} we get
		
		\begin{align*}
			\Delta h(u)&=h'(u)\Delta u+h''(u)|\nabla u|^2\\[1mm]
			&=h'(u)\left(\psi u+\frac{|\nabla u|^2}{u}\right)+h''(u)|\nabla u|^2.
		\end{align*}
		From \eqref{Deltalog2} let us observe that 
		\[\psi u=\Tilde{\psi} u+\frac{f(\tau)^2}{u}\text{tr}(A^{2}_{T^{N}})\]
		where 
		\[\Tilde{\psi} =2\left(K_\Sigma-\frac{f'(\tau)^{2}+K _ {M}}{f(\tau)^{2}}\right)+\frac{f(\tau)^{2}}{u}\sum_{i=1}^{n-3}\text{tr}(A^{2}_{\xi_{i}}) \geq 0.\]
		Here we considered  assumption \eqref{Kgeqzero}. Note that by \eqref{gradu2} we have
		\[|\nabla u|^{2} = 2\text{tr}(A^{2}_{T^{N}})(u-f(\tau)^2) \]
		with $|T^{\top}|^{2} = u -f(\tau)^{2}$. Therefore we obtain the following expression 
		\begin{equation}\label{new2}
			\Delta h(u)=h'(u)\Tilde{\psi}u+\frac{h'(u)\text{tr}(A^{2}_{T^{N}})}{u}\left(f(\tau)^2+2(u-f(\tau)^2)+\frac{2uh''(u)}{h'(u)}(u-f(\tau)^2)\right).
		\end{equation}
		Note that, $h'(u)\leq 0$, so if $f(\tau)^2+2(u-f(\tau)^2)+\dfrac{2uh''(u)}{h'(u)}(u-f(\tau)^2)$ is non-negative,  we have that $\Delta h(u)\leq 0,$ i.e., the non-negative function $h(u)$ on $\Sigma^2$ is superharmonic.\\
		
		{\bf Claim:} Since $h(t) = (1+t)^{-\alpha}$, if we denote  \[\theta:=f(\tau)^2+2(u-f(\tau)^2)+\frac{2uh''(u)}{h'(u)}(u-f(\tau)^2), \] then $\theta \geq 0$ for some  $0<\alpha\leq 1 $. \\
		
		In fact we have
		\begin{align}
			\theta &= f(\tau)^2+2(u-f(\tau)^2)-\frac{2u(u-f(\tau)^2)(\alpha+1)}{1+u}\\
			&=\frac{1}{1+u}[f(\tau)^2(1+u)+2(u-f(\tau)^2)(1+u)-2u(u-f(\tau)^2)(\alpha+1)].
		\end{align}
		Thus, it is sufficient to prove that
		\[f(\tau)^2(1+u)+2(u-f(\tau)^2)(1+u)-2u(u-f(\tau)^2)(\alpha+1)\geq 0, \]
		i.e.,
		\begin{align}
			-f(\tau)^2+(2\alpha+1)f(\tau)^2u+2u-2u^2\alpha \geq 0,\\[1mm]
			(u-f(\tau)^2)+(2\alpha +1)f(\tau)^2u+u-2u^2\alpha \geq 0.
		\end{align}
		Since $u-f(\tau)^2\geq 0$, we just need to prove that 
		\begin{equation}
			(2\alpha +1)f(\tau)^2u+u-2u^2\alpha \geq 0,
		\end{equation}   
		or equivalently
		\begin{equation}\label{eq bound of u}
			u\leq \frac{2\alpha+1}{2\alpha}f(\tau)^2+\frac{1}{2\alpha}
		\end{equation}
		that holds from our hypothesis \eqref{subquadratic} if we choose $\alpha>0$ small enough. Therefore, for $u\leq C f(\tau)^2+C$, where the constant $C= \dfrac{1}{2\alpha}$, we have \eqref{eq bound of u}, whence $\theta \geq 0$. Hence $h(u)$ is superharmonic. By \eqref{Kgeqzero}, using a classical result by Ahlfors and Blanc-Fiala-Hubber \cite{K:87} that asserts that a complete Riemannian surface with non-negative Gauss curvature must be parabolic, we have $h(u)$ is a constant,  which implies $u$ is also constant, whence $\Sigma^2$ is totally geodesic.
	\end{proof}
	
	When the function $u$ is bounded we have the following consequence. Although this result follows directly from Theorem \ref{teo 3.2} we give an alternative proof of a local character. 
	\begin{corollary}\label{cor1}
		Let $x:\Sigma^2\to \overline{M}^{n}$ be a complete stationary spacelike surface in a GRW spacetime $\overline{M}^{n}=-I\times_fM^{n-1}$, such that $u = - \langle T^N,T^N\rangle$ is bounded on $\Sigma^2$. If
		\begin{equation}\label{eq 3}
			K_\Sigma\geq \frac{f'(\tau)^{2}+K _ {M}}{f(\tau)^{2}}\geq 0.
		\end{equation}
		then $\Sigma^2$ must be totally geodesic.
	\end{corollary} 
	\begin{proof}
		From \eqref{eq 3}, Lemma \ref{Deltalog2} gives
		\begin{equation}
			\Delta \log u \geq \frac{f(\tau)^{2}}{u^{2}}\left(\text{tr}(A^{2}_{T^{N}})+u\sum_{i=1}^{n-3}\text{tr}(A^{2}_{\xi_{i}})\right) \geq 0,
		\end{equation}
		and therefore
		\begin{equation} 
			\Delta u \geq \frac{f(\tau)^{2}}{u}\left(\text{tr}(A^{2}_{T^{N}})+u\sum_{i=1}^{n-3}\text{tr}(A^{2}_{\xi_{i}})\right)+\frac{|\nabla u|^2}{u}\geq 0.
		\end{equation}              
		Therefore, $u\Delta u \geq 0$ and using Lemma \ref{Lem 2.2}  we arrive to the following local integral estimation of the length of the gradient of $u$ 
		\begin{equation}\label{new3}
			\int_{B_r}|\nabla u|^{2}dV\leq \frac{4\sup_{B_{R}}u^{2}}{\mu_{r,R}}\leq \frac{4\sup_{\Sigma}u^{2}}{\mu_{r,R}} \,,
		\end{equation}
		where $R$ is the radius of a geodesic ball $B_R$ in $\Sigma^2$ centered at a point $p\in \Sigma^2$, and $B_r$, $0<r<R$, is a geodesic ball centered at the same point $p$ and contained in $B_R$. 
		
		Note that \eqref{eq 3} also says that the Gauss curvature of $\Sigma^2$ is nonnegative. Hence, the Ahlfors and Blanc-Fiala-Hubber Theorem can be used again to conclude that $\Sigma^2$ is parabolic. In this case, we know  
		$\lim_{R\rightarrow\infty} \dfrac{1}{\mu_{r,R}} = 0$, accordingly to  \cite{K:87}. Now the proof ends taking into account \eqref{new3} which implies that $\nabla u(p)=0$ for any $p\in \Sigma^2$, i.e., the function $u$ is constant on $\Sigma^2$   
		which implies that the equality holds in formula \eqref{eq K}. Therefore, $\Sigma^2$ must totally geodesic from formula \eqref{Deltalog eq}.
	\end{proof}
	
	\begin{corollary}
		Let $x:\Sigma^2\to \overline{M}^{n}$ be a complete stationary spacelike surface in a static GRW spacetime $\overline{M}^{n}= -I \times M^{n-1}$ such that $u=-\langle(\partial/\partial t)^N,(\partial/\partial t)^N\rangle$  (or equivalently $| \nabla \tau |$) is bounded on $\Sigma^2$. If 
		$K_\Sigma\geq {K _ {M}}\geq 0\label{Kgeqzero2}$, then $\Sigma^2$ must be totally geodesic.
	\end{corollary}
	\hyphenation{pa-ra-bo-li-ci-ty}
	\begin{remark}
		The assumption $f'(\tau)^{2}+K_{M}\geq 0$ is only used at the ends of the proofs of Theorem \ref{teo 3.2} and Corolary \ref{cor1} to assert $K_\Sigma\geq 0$ because $K_\Sigma\geq \dfrac{f'(\tau)^{2}+K _ {M}}{f(\tau)^{2}}$. However, we arrive to the same conclusion in these results if we change the assumption \eqref{Kgeqzero} to
		\[
		K_{\Sigma}\geq \frac{f'(\tau)^{2}+K_{M}}{f(\tau)^{2}}\ \ \text{and}\ \ K_\Sigma\geq 0. 
		\]
		On the other hand, taking into account that the assumption $K_\Sigma\geq 0$ is only needed to assert that $\Sigma^2$ is parabolic, the same conclusion of Theorem \ref{teo 3.2} and Corolary \ref{cor1} is obtained, if we changed the assumption \eqref{Kgeqzero} by
		\[K_{\Sigma}\geq \frac{f'(\tau)^{2}+K_{M}}{f(\tau)^{2}}\ \ \text{and}\ \ \Sigma^2 \ \text{is parabolic}. \]
		Moreover, the parabolicity of $\Sigma^2$ can be also derived from other curvature assumptions which are well known in the literature \cite{K:87}, \cite{Li:2000}.
	\end{remark}
	Formula \eqref{Delta u} suggests  that other assumptions on the curvature  together with the parabolicity of the stationary surface implies its rigidity in the GRW spacetime. For the case that the function $u$ is controlled by the squared warping function and the surface has Gauss curvature nonnegative we have the following result,      
	
	\begin{theorem}\label{teo 2}
		Let $x:\Sigma^2\to \overline{M}^{n}$ be a complete stationary spacelike surface in a GRW spacetime $\overline{M}^{n}=-I\times_fM^{n-1}$. If     
		\vspace{1mm}
		\begin{equation}\label{new4}
			K_{\Sigma}\geq\max \left\{\,\frac{f''(\tau)}{f(\tau)},\, 0\,\right\} 
		\end{equation}
		and the function $u = - \langle T^{N},T^{N}\rangle$ on $\Sigma^2$ satisfies   
		\begin{equation}
			u \leq Cf(\tau)^2+C,  
		\end{equation} 
		\vspace{1mm}
		for some positive constant $C$, then 
		\begin{equation}
			K_{\Sigma} \leq \frac{f'(\tau)^{2}+K _ {M}}{f(\tau)^{2}}
		\end{equation}
		and equality holds if, and only if, $\Sigma^2$ is totally geodesic.
	\end{theorem}
	
	\begin{proof}
		Let $u=- \langle T^{N},T^{N}\rangle$ and consider the function $v = \dfrac{-1}{(1+u)^{\alpha}}$ on $\Sigma^2$, where $\alpha >0$ is constant. Note that $-1< v < 0$, in particular, it is bounded from above. From \eqref{new4} we have $K_{\Sigma}\geq 0$, and thus, the  classical result by Ahlfors and Blanc-Fiala-Hubber \cite{K:87} ensures that $\Sigma^2$ is parabolic. Let us show now that $v$ is subharmonic.\\ 
		
		We compute first the gradient of $v$. For any $X\in \mathfrak{X}(\Sigma)$, we have
		
		\begin{align*}
			\left \langle \nabla\left(\frac{-1}{(1+u)^{\alpha}}\right),X\right\rangle &= X\left(\frac{-1}{(1+u)^{\alpha}}\right) = \frac{\alpha}{(1+u)^{\alpha+1}}X(u)\\[1mm]
			&=\left \langle\frac{\alpha \nabla u}{(1+u)^{\alpha +1}},X\right\rangle.
		\end{align*}
		Therefore,
		\begin{equation}\label{new5}
			\nabla v = \frac{\alpha}{(1+u)^{\alpha+1}}\,\nabla u
		\end{equation}      
		Next, using the well-known formula  $\text{div}(hX) = X(h)+h\,\text{div}(X)$, for any  smooth function $h$ on $\Sigma^2$ and any $X\in \mathfrak{X}(\Sigma)$ and \eqref{new5} we obtain
		\begin{equation}\label{new6}        
			\Delta v = \text{div} \left(\frac{\alpha}{(1+u)^{\alpha +1}}\nabla u\right) = \nabla u\left(\frac{\alpha}{(1+u)^{\alpha+1}}\right)+\frac{\alpha}{(1+u)^{\alpha+1}}\,\Delta u\,.  
		\end{equation}
		Taking into account that
		\begin{align*}
			\nabla u\left(\frac{\alpha}{(1+u)^{\alpha+1}}\right) &=  -\alpha(\alpha+1)(1+u)^{-\alpha-2}\,(\nabla u)(u)\\[1mm]
			&= - \frac{\alpha(\alpha+1)}{(1+u)^{\alpha+2}}|\nabla u|^{2},
		\end{align*}
		formula \eqref{new6} can be rewritten as follows\\
		\begin{equation}\label{new7} 
			\Delta v  = -\frac{\alpha(\alpha+1)}{(1+u)^{\alpha+2}}|\nabla u|^{2}+\frac{\alpha}{(1+u)^{\alpha+1}}\Delta u 
		\end{equation}\\
		By using \eqref{gradu2} and \eqref{Delta u}, formula \eqref{new6} gives
		\begin{align*}
			\Delta v &=\frac{2\alpha}{(1+u)^{\alpha+1}}\left(K-\frac{f(\tau)''}{f(\tau)}\right)(u-f(\tau)^2)&\\ &+\frac{2\alpha}{(1+u)^{\alpha+1}} \text{tr}(A^{2}_{T^{N}})\left[ 1-\frac{\alpha+1}{1+u}(u-f(\tau)^2)\right]&\\[1.5mm]
			&=\frac{2\alpha}{(1+u)^{\alpha+1}}\left(K-\frac{f(\tau)''}{f(\tau)}\right)(u-f(\tau)^2)&\\[1.5mm]
			&+\frac{2\alpha}{(1+u)^{\alpha+1}} \text{tr}(A^{2}_{T^{N}})\left[\frac{1-\alpha u+\alpha f(\tau)^2+f(\tau)^{2}}{1+u}\right].
		\end{align*}\\
		Note that, in order to prove that  $\Delta v\geq 0$ from previous formula, it remains to show that 
		$$
		1-\alpha u+\alpha f(\tau)^2+f(\tau)^{2}\geq 0,
		$$      
		or equivalently
		\[
		u\,\leq\, \frac{\alpha+1}{\alpha}f(\tau)^2+\dfrac{1}{\alpha}, 
		\]
		which holds true if we choose $\alpha=\dfrac{1}{C}$\,.\\
		
		Hence, the function $v<0$ satisfies $\Delta v\geq 0$ for this choice of $\alpha$. Therefore, $v$ is constant from the   parabolicity of $\Sigma^2$. From \eqref{new5}, $u$ is also constant. Formula \eqref{Deltalog eq} gives the inequality
		\[K_{\Sigma}\leq \frac{f'(\tau)^{2}+K_{M}}{f(\tau)^{2}}, \]
		and the equality holds if, and only if, $\Sigma^2$ is totally geodesic.
	\end{proof}
	
	In particular, we have the following result for the  case $u$ is bounded. 
	
	\begin{corollary}
		Let $x:\Sigma^2\to \overline{M}^{n}$ be a complete stationary spacelike surface in a GRW spacetime $\overline{M}^{n}=-I\times_fM^{n-1}$, such that $u = - \langle T^N,T^N\rangle$ is bounded on $\Sigma^2$. If
		
		\[K_{\Sigma}\geq\max \left\{\frac{f''(\tau)}{f(\tau)}, 0\right\} \]
		then
		\[K_{\Sigma} \leq \frac{f'(\tau)^{2}+K _ {M}}{f(\tau)^{2}} \]\\      
		and equality holds if, and only if, $\Sigma^2$ is totally geodesic. 
	\end{corollary}
	
	\section{Examples}
	\begin{example}\label{Ex1}
		Consider the four dimensional Lorentz-Minkowski spacetime $\mathbb{L}^{4}$, i.e., $\mathbb{L}^{4} = \R^{4}$ endowed with $-dx_{1}^{2}+dx_{2}^{2}+dx_{3}^{2}+dx_{4}^{2}.$ For each $w:\R^{2}\rightarrow \R$, harmonic and non-constant, the mapping
		\[\psi_{w}:\R^{2}\rightarrow \mathbb{L}^{4},\quad  
		\psi_{w}(x,y) = (w(x,y),x,y,w(x,y)) \]
		is clearly a spacelike immersion  whose induced metric on $\R^2$ is just the Euclidean one $dx^{2}+dy^{2}$, for any $w$. Obviously, the mean curvature vector field of $\psi_w$ is identically zero. Thus, $\psi_w : \R^2 \to \L^4$ is a stationary spacelike surface for any $w$, and it is totally geodesic if and only if $w(x,y)=ax+by+c$, for $a,b,c\in \R$.  
		
		For $\partial_{t} (= \partial x_1) = (1,0,0,0)$, write along $\psi_{w}$
		as $\partial_{t} = \partial_{t}^{\top}+\partial_{t}^{N}$ and then $-1 = \langle \partial_{t}^{\top},\partial_{t}^{\top}\rangle +\langle \partial_{t}^{N}, \partial_{t}^{N}\rangle $. We want to compute $\langle \partial_{t}^{N},  \partial_{t}^{N}\rangle .$ First we have
		\[\partial_{t}^{\top} = \Big\langle\partial_{t}^{\top},\frac{\partial \psi_{w}}{\partial x}\Big\rangle\frac{\partial \psi_{w}}{\partial x}+\Big\langle\partial_{t}^{\top},\frac{\partial \psi_{w}}{\partial y}\Big\rangle\frac{\partial \psi_{w}}{\partial y} \,,
		\]
		where
		\[
		\frac{\partial \psi_{w}}{\partial x} = \left(\frac{\partial w}{\partial x},\,1,\,0,\frac{\partial w}{\partial x}\right)\quad \text{and}\quad \frac{\partial \psi_{w}}{\partial y} = \left(\frac{\partial w}{\partial y},0\,,1\,,\frac{\partial w}{\partial y}\right) \]
		are orthonormal everywhere. Taking into account
		\[
		\Big\langle \partial_{t}^{\top},\frac{\partial \psi_{w}}{\partial x}\Big\rangle  =\Big\langle \partial_{t},\frac{\partial \psi_{w}}{\partial x}\Big\rangle = -\frac{\partial w}{\partial x}\quad \text{and}\quad \Big\langle \partial_{t}^{\top},\frac{\partial \psi_{w}}{\partial y}\Big\rangle  =\Big\langle \partial_{t},\frac{\partial \psi_{w}}{\partial y}\Big\rangle = -\frac{\partial w}{\partial y},
		\] 
		we have 
		\[
		\langle \partial_{t}^{\top}, \partial_{t}^{\top}\rangle  = \left(\frac{\partial w}{\partial x}\right)^{2}+\left(\frac{\partial w}{\partial y}\right)^{2}. 
		\]
		In the particular case, $w(x,y) = x^2-y^2,$ we have $\langle \partial_{t}^{\top}, \partial_{t}^{\top}\rangle  = 4x^{2}+4y^{2}$. Therefore, 
		\[
		u=-\langle \partial_{t}^{N}, \partial_{t}^{N}\rangle = 1+4x^{2}+4y^{2}\rightarrow \infty\quad \text{as}\quad x^2+y^2 \rightarrow \infty\,, 
		\]
		i.e., the function $u$ on $\R^2$ is unbounded. Therefore, assumption \eqref{subquadratic} in Theorem \ref{teo 3.2}  cannot be dropped (note that formula \eqref{Kgeqzero} is obviously satisfied because $K_{\R^2}=0$).
	\end{example}
	
	\begin{example}\label{ex2}
		Consider the holomorphic functions $\phi_{1}(z) = 1-z^{2},\ \phi_{2}(z) = i(1+z^{2}),\ \phi_{3}(z) = 2z$ and $\phi_{4}(z) = \sqrt{2}(1-z^{2}),\ z\in\mathbb{C}$ that satisfy
		\begin{equation}\label{phi1}
			-\phi_{1}^{2}(z)+\phi_{2}^{2}(z)+\phi_{3}^{2}(z)+\phi_{4}^{2}(z) = 0,
		\end{equation}
		\begin{equation}\label{phi2}
			-|\phi_{1}(z)|^{2}+|\phi_{2}(z)|^{2}+|\phi_{3}(z)|^{2}+|\phi_{4}(z)|^{2}=2(1+|z|^2)^2>0.
		\end{equation}
		for any $z\in \mathbb{C}$.  
		
		The map 
		\[
		\psi : \mathbb{C}\rightarrow \mathbb{L}^{4},\quad  \psi(z)= \text{Real}\int \big(1-z^{2},i(1+z^{2}),2z,\sqrt{2}(1-z^{2})\big)dz ,
		\]
		defines a stationary spacelike surface in $\L^4$ thanks to \eqref{phi1} and \eqref{phi2}, \cite{ER}, and the induced metric on $\mathbb{C}$ is, 
		\begin{equation}\label{induced metric}
			g:=(1+|z|^2)^2\,|dz|^2.
		\end{equation}
		If we put $z=x+iy$ then we can equivalently describe $\psi$ as follows
		\begin{equation}\label{immersion2}
			\psi(x,y) = \Big(x-\frac{x^{3}}{3}+xy^{2},-y-x^{2}y+\frac{y^{3}}{3},x^{2}-y^{2},\sqrt{2}x-\frac{\sqrt{2}}{3}x^{3}+\sqrt{2}xy^{2}\Big), 
		\end{equation}
		for all $(x,y)\in\R^2$.
		
		Next, we prove that the Riemannnian metric \eqref{induced metric} is complete. Recall that for a non-compact Riemannian manifold, completeness is equivalent to the following property \cite[p. 153]{DoCarmo}: every divergent curve starting at any point, has infinite length, i.e., in our case, for any (smooth) curve $\gamma : [0,\infty) \rightarrow \mathbb{C}$, such that for every compact subset $C$ of $\mathbb{C}$ there exists $t_0 \in (0,\infty)$ such that $\gamma (t) \notin C$ for any $t>t_0$, we have
		\[   
		\lim_{T \to \infty}\int_0^{T}\sqrt{g(\gamma{\,'}(t),\gamma{\,'}(t))}\,dt=\infty\,.
		\]
		But this is true because the euclidean length of a divergent curve in $\mathbb{C}$ is $\infty$, and 
		\[
		\int_0^{T}\sqrt{g(\gamma{\,'}(t),\gamma{\,'}(t))}\,dt \geq \int_0^{T}|\gamma{\,'}(t)|\,dt
		\]
		holds true, making use of \eqref{induced metric}.\\
		
		On the other hand, being the induced metric \eqref{induced metric} pointwise conformal to the Euclidean one of $\mathbb{C}$, its Gauss curvature, $K$, satisfies
		\begin{equation}\label{Gauss curvature}
			K(z) = -\frac{1}{(1+|z|^2)^2}\Delta\log (1+|z|^2)=\frac{-4}{(1+|z|^2)^4}<0,
		\end{equation}
		where $\Delta=\dfrac{\partial^2}{\partial x^2}+\dfrac{\partial^2}{\partial y^2}$ is the usual Laplacian. Therefore, $\psi$ is not totally geodesic.\\ 
		
		Finally,  we compute $\langle \partial_{t}^{N},  \partial_{t}^{N}\rangle .$ As in Example \ref{Ex1} we have
		\[
		\partial_{t}^{\top} = \Big\langle\partial_{t}^{\top},\frac{1}{\lambda}\frac{\partial \psi}{\partial x}\Big\rangle\frac{1}{\lambda}\frac{\partial \psi}{\partial x}
		+\Big\langle\partial_{t}^{\top},\frac{1}{\lambda}\frac{\partial \psi}{\partial y}\Big\rangle\frac{1}{\lambda}\frac{\partial \psi}{\partial y} \,,
		\]
		where $\lambda(x,y)=1+x^2+y^2$. Using now \eqref{immersion2} we get 
		\[
		\Big\langle \partial_{t}^{\top},\frac{\partial \psi}{\partial x}\Big\rangle  =\Big\langle \partial_{t},\frac{\partial \psi}{\partial x}\Big\rangle = -(1-x^2+y^2)\quad \text{and}\quad \Big\langle \partial_{t}^{\top},\frac{\partial \psi}{\partial y}\Big\rangle  =\langle \partial_{t},\frac{\partial \psi}{\partial y}\Big\rangle = -2xy,
		\] 
		and therefore
		\[
		\langle \partial_{t}^{\top}, \partial_{t}^{\top}\rangle  = \left(\frac{1-x^2+y^2}{1+x^2+y^2}\right)^{2}+\left(\frac{2xy}{1+x^2+y^2}\right)^{2}. 
		\]
		Thus, we obtain 
		\[
		u=-\langle \partial_{t}^{N}, \partial_{t}^{N}\rangle = 1+\langle \partial_{t}^{\top}, \partial_{t}^{\top}\rangle= 2\,\frac{(1+x^2+y^2)^2-2x^2}{(1+x^2+y^2)^2}\leq 2,
		\]
		for any $(x,y)\in \R^2$, i.e., the function $u$ on $\R^2$ is bounded. Therefore, assumption  \eqref{Kgeqzero}  in Theorem \ref{teo 3.2} cannot be dropped, although \eqref{subquadratic} is satisfied as this example shows.
		
	\end{example}
	\section*{Acknowledgements}
	
The first author was partially supported by CAPES//Brazil. The second author was partially supported by CNPq, Brazil, Universal Grant 438601/2018-1, and the third named author has been partially supported by the Spanish MICINN and ERDF project PID2020-116126GB-I00, and by the Andalusian and ERDF project A-FQM-494-UGR18.

Danilo Ferreira and Eraldo A. Lima Jr.\\
{\sc Departamento de Matemática, Universidade Federal da Paraíba, 58051-900 João Pessoa-PB Brazil}.\\
\textit{Email-address: danilo.silva@academico.ufpb.br}\\
\textit{Email-address: eraldo.lima@academico.ufpb.br}\\

 Alfonso Romero\\
{\sc Departamento de Geometría y Topología, Universidad de Granada, 18071 Granada
Spain}\\
\textit{Email-address: aromero@ugr.es}

\end{document}